\documentclass[a4paper,12pt]{article}

\usepackage{amsmath,amsfonts,amssymb}
\usepackage{amsthm}

\newtheorem{lemma}{Lemma}[section]
\newtheorem{theorem}[lemma]{Theorem}
\newtheorem{corollary}[lemma]{Corollary}

\newtheorem{remark}[lemma]{Remark}
\newtheorem{definition}[lemma]{Definition}

\date{} 
\begin{document}

\title{On the essential norms of Toeplitz operators with continuous symbols}

\author{Eugene Shargorodsky}

\maketitle

\begin{abstract}
It is well known that the essential norm of a Toeplitz operator on the Hardy space $H^p(\mathbb{T})$, $1 < p < \infty$ is greater than or equal to
the $L^\infty(\mathbb{T})$ norm of its symbol. In 1988, A. B\"ottcher, N. Krupnik, and B. Silbermann posed a question on whether or not the 
equality holds in the case of continuous symbols. We answer this question in the negative. On the other hand, we show that the essential norm of 
a Toeplitz operator with a continuous symbol is less than or equal to twice the $L^\infty(\mathbb{T})$ norm of the symbol and prove more precise $p$-dependent
estimates. 
\end{abstract}


\section{Introduction}

For Banach spaces $X$ and $Y$, let $\mathcal{B}(X, Y)$ and $\mathcal{K}(X, Y)$ denote the sets of bounded linear and compact linear operators 
from $X$ to $Y$, respectively.

For $A \in \mathcal{B}(X, Y)$, let
$$
\text{Ker} \, A := \{x \in X | \ Ax = 0\} ,
\ \ \
\text{Ran}\, A := \{Ax | \ x \in X\} .
$$

The operator $A$\, is called Fredholm if
$$
\text{dim} \, \text{Ker}\, A < +\infty , \ \
\text{dim} \, (X/\text{Ran}\, A) < +\infty .
$$

The essential spectrum of $A \in \mathcal{B}(X) := \mathcal{B}(X, X)$ is the set
$$
\text{\rm Spec}_{\rm e}(A) := \{\lambda \in \mathbb{C} : \ A - \lambda I \
\text{ is not Fredholm}\} .
$$
The essential norm of $A \in \mathcal{B}(X, Y)$ is defined as follows:
$$
\|A\|_{\mathrm{e}} := \inf\{\|A - K\| : \  K \in \mathcal{K}(X, Y)\} .
$$
For any $A \in \mathcal{B}(X)$, \ $\text{\rm Spec}_{\rm e}(A)$ and $\|A\|_{\mathrm{e}}$ are equal to the spectrum and the norm
of the corresponding element $[A]$ of the Calkin algebra $\mathcal{B}(X)/\mathcal{K}(X)$ (see, e.g., \cite[Sect. 4.3]{D07} or \cite[Sect. XI.5]{GGK90}).
Hence the essential spectral radius of $A \in \mathcal{B}(X)$ is less than or equal to its essential norm:
\begin{equation}\label{essineq}
r_{\rm e}(A) := \sup\left\{|\lambda| : \ \lambda \in \text{\rm Spec}_{\rm e}(A)\right\} \le \|A\|_{\mathrm{e}} .
\end{equation}

Let $\mathbb{T}$ be the unit circle: $\mathbb{T} :=\left\{z \in \mathbb{C} : \ |z| = 1\right\}$. For a function $f \in L^1(\mathbb{T})$, let
$$
\widehat{f}(k) = \frac{1}{2\pi} \int_{-\pi}^\pi f\left(e^{i\theta}\right) e^{-i k\theta}\, d\theta , \ \ \ k \in \mathbb{Z}
$$
be the Fourier coefficients of $f$. The Hardy spaces, the Riesz projection, and 
the Toeplitz operator with the symbol $a \in L^\infty(\mathbb{T})$ are defined in the usual way:
\begin{align*}
& H^p(\mathbb{T}) := \left\{f \in L^p(\mathbb{T}) : \ \widehat{f}(k) = 0 \ \mbox { for all } \ k < 0\right\} , \ \ \ 1 \le p \le \infty , \\
& (P f)\left(e^{i\theta}\right) := \sum_{k \ge 0} \widehat{f}(k) e^{i k\theta} \ \ \mbox{ for } \ \ 
f\left(e^{i\theta}\right) = \sum_{k = -\infty}^\infty \widehat{f}(k) e^{i k\theta} , \\
&  T(a) f := P(a f) , \ \ \ f \in L^1(\mathbb{T}) .
\end{align*}
If $1 < p < \infty$, the Riesz projection $P : L^p(\mathbb{T}) \to H^p(\mathbb{T})$ is bounded (see, e.g., \cite[Ch. 9]{H62}) and hence the
Toeplitz operator 
\begin{equation}\label{T}
T(a) = PaI : H^p(\mathbb{T}) \to H^p(\mathbb{T}) , \ 1 < p < \infty , \quad a \in L^\infty(\mathbb{T})
\end{equation}
is bounded. Everywhere in the paper, apart from Section \ref{concl}, $T(a)$ denotes operator \eqref{T}.

Since
\begin{equation}\label{HWS}
a(\mathbb{T})_{\rm e} := \left\{\lambda \in \mathbb{C} : \ \frac1{a - \lambda} \not\in L^\infty(\mathbb{T})\right\} \subseteq \text{\rm Spec}_{\rm e}(T(a)) 
\end{equation}
(see, e.g., \cite[Theorem 2.30]{BS06}), inequality \eqref{essineq} implies
\begin{equation}\label{essnr}
\|a\|_{L^\infty} \le r_{\rm e}(T(a)) \le  \|T(a)\|_{\mathrm{e}} .
\end{equation}
On the other hand,
$$
\|T(a)\|_{\mathrm{e}} \le \|T(a)\| =\|P a\,\mathrm{I}\| \le \|P\| \|a\|_{L^\infty} .
$$
Since
$$
\|P\|_{L^p \to L^p} = \frac{1}{\sin\frac{\pi}{p}}
$$
(see \cite{HV00}), one gets
\begin{equation}\label{2sided}
\|a\|_{L^\infty} \le  \|T(a)\|_{\mathrm{e}} \le \frac{1}{\sin\frac{\pi}{p}} \|a\|_{L^\infty} .
\end{equation}
If $p = 2$, inequality \eqref{2sided} turns into the equality $\|T(a)\|_{\mathrm{e}} = \|a\|_{L^\infty}$.
If $a \equiv 1$, then $\|a\|_{L^\infty} = 1 =  \|T(a)\|_{\mathrm{e}}$, so the first inequality in \eqref{2sided} is sharp. If
$$
a\left(e^{i\vartheta}\right) := \sin\frac{\pi}{p} \pm i \cos\frac{\pi}{p} , \ \ \ 
\pm \vartheta \in (0, \pi) ,
$$
then $\|a\|_{L^\infty} = 1$, and it follows from the Gohberg-Krupnik theory of  of Toeplitz operators with piecewise continuous symbols that
$$
\frac{1}{\sin\frac{\pi}{p}} \in \text{\rm Spec}_{\rm e}(T(a))  
$$
(see, e.g., \cite[Theorem 5.39]{BS06}). Hence
$$
\|T(a)\|_{\mathrm{e}} \ge \frac{1}{\sin\frac{\pi}{p}}
$$
(see \eqref{essineq}), and the second inequality in \eqref{2sided} is also sharp if one considers Toeplitz operators with discontinuous symbols. 

The situation is different in the case of continuous symbols. 
If $a \in C(\mathbb{T})$, then $\text{\rm Spec}_{\rm e}(T(a)) = a(\mathbb{T})$
(see, e.g., \cite[Theorem 2.42]{BS06}). In particular, $\text{\rm Spec}_{\rm e}(T(a))$ does not depend on $p$. It is natural to ask whether 
$\|T(a)\|_{\mathrm{e}}$ depends on $p$ for $a \in C(\mathbb{T})$. Since $\|T(a)\|_{\mathrm{e}} = \|a\|_{L^\infty}$ for $p = 2$, this question can be rephrased as follows:
does the equality $\|T(a)\|_{\mathrm{e}} = \|a\|_{L^\infty}$ hold for all $p \in (1, \infty)$ and all $a \in C(\mathbb{T})$? This question was posed in 
\cite{BKS}, where it was proved that
$$
 \|T(a)\|_{\mathrm{e}} = \|a\|_{L^\infty} \mbox{ for all } a \in (C + H^\infty)(\mathbb{T}) \
 \Longleftrightarrow \ \|T(\mathbf{e}_{-1})\|_{\mathrm{e}} = 1 .
$$
Here and below,
\begin{equation}\label{em}
\mathbf{e}_m(z) := z^m\,  , \quad z \in \mathbb{C} , \  m \in \mathbb{Z} .
\end{equation}
Note that for every $f \in  H^p(\mathbb{T})$,
$$
f\left(e^{i\theta}\right) =  \sum_{n = 0}^\infty \widehat{f}(n)e^{in\theta} , \quad \theta \in  [-\pi, \pi] ,
$$
one has
$$
\mathbf{e}_{-1}\left(e^{i\theta}\right)f\left(e^{i\theta}\right) = 
\widehat{f}(0)e^{-i\theta} + \sum_{n = 1}^\infty \widehat{f}(n)e^{i(n - 1)\theta} ,
$$
and hence
$$
\big(T(\mathbf{e}_{-1})f\big)\left(e^{i\theta}\right) = \sum_{n = 1}^\infty \widehat{f}(n)e^{i(n - 1)\theta} .
$$
If $\widehat{f}(0) = 0$, then
$$
T(\mathbf{e}_{-1})f = \mathbf{e}_{-1}f \quad \Longrightarrow \quad |T(\mathbf{e}_{-1})f| = |f| \ \mbox{ a.e. on } \ \mathbb{T} .
$$
So, the equality $\|T(\mathbf{e}_{-1})f\|_{H^p} = \|f\|_{H^p}$ holds on a co-dimension one subspace of $H^p(\mathbb{T})$, and the 
equality $\|T(\mathbf{e}_{-1})\|_{\mathrm{e}} = 1$ looks plausible. Nevertheless, we show that the answer to the above question is negative 
and $\|T(\mathbf{e}_{-1})\|_{\mathrm{e}} > 1$ for every $p \not= 2$ (see Section \ref{model}).

The constant $\frac{1}{\sin\frac{\pi}{p}}$ in the right-hand side of \eqref{2sided} tends to infinity as $p \to 1$ or $\infty$. It turns out
that a better estimate holds for $T(a) : H^p(\mathbb{T}) \to H^p(\mathbb{T})$, $1 < p < \infty$\, if
$a \in (C + H^\infty)(\mathbb{T})$. Namely,
$$
\|T(a)\|_{\mathrm{e}}  \le 2^{\left|1 - \frac2p\right|}\|a\|_{L^\infty}  
$$
(see Section \ref{2pp'}). This estimate implies that
$$
\|T(a)\|_{\mathrm{e}}  \le 2\|a\|_{L^\infty} 
$$
for every $p \in (1, \infty)$, and we show in Section \ref{concl} that the latter can be extended from $H^p(\mathbb{T})$
to a much wider class of abstract Hardy spaces built upon Banach function spaces.

The proof of our main results relies upon the use of measures of noncompactness (see Section \ref{noncomp}) and approximation
properties of Hardy spaces (see Section \ref{bcaph}).

Appendix contains some well known results on adjoints to restrictions of operators for which we could not find a convenient reference.

\section{Measures of noncompactness of a linear operator}\label{noncomp}

For a bounded subset $\Omega$ of a Banach space $Y$, we 
denote by $\chi(\Omega)$
the greatest lower bound of the set of numbers $r$ such that $\Omega$ can 
be covered by a finite family of open balls of radius $r$.

For $A \in \mathcal{B}(X, Y)$, set
$$
\|A\|_\chi := \chi\left(A(B_X)\right) ,
$$  
where $B_X$ denotes the unit ball in $X$. 
Let 
$$
\|A\|_m := \inf_{{\tiny \begin{array}{c}
M \subseteq X \mbox{ closed linear subspace}   \\
\mathrm{dim}  (X/M) < \infty  
\end{array} }} \big\|A|_M\big\| ,
$$
where $A|_M$ denotes the restriction of $A$ to $M$.

\begin{remark}\label{equal}
{\rm The following measure of noncompactness introduced by H.-O. Tylli proved to be convenient in estimating the
essential norms of pointwise multipliers in function spaces (see \cite{ES05}):
\begin{align*}
\beta(A) := \inf\{\rho > 0 :  &\mbox{ there exist a Banach space } Z  \mbox{ and }  R \in \mathcal{K}(X, Z) \nonumber \\
&\mbox{ such that } 
\|Ax | Y\| \le \rho \|x | X\| + \|Rx | Z\| , \
\forall x \in X\} .
\end{align*} 
It was shown in \cite{GM91} that $\beta(A) = \|A\|_m$.} 
\end{remark}

The measures of noncompactness $\|\cdot\|_\chi$ and $\|\cdot\|_m$ have the following properties
\begin{equation}\label{nontrivial}
\frac12\, \|A\|_\chi \le \|A\|_m \le 2 \|A\|_\chi
\end{equation}
and
\begin{equation}\label{trivial}
\|A\|_\chi \le \|A\|_{\mathrm{e}}\, , \quad \|A\|_m \le \|A\|_{\mathrm{e}}
\end{equation}
(see \cite{LS}; note that there is a typo in \cite[(3.7)]{LS}, where the factor $2$ is missing in the right-hand side).
The constants $\frac12$ and $2$ in \eqref{nontrivial} are optimal (see \cite[2.5.2 and 2.5.6]{AKPRS92}).

In general, $\|A\|_{\mathrm{e}}$ cannot be estimated above by $\|A\|_\chi $ or $\|A\|_m$ (see \cite{AT87}). Some restrictions on the geometry of $X$ or $Y$
are needed for such estimates. 
\begin{definition}\label{def1} A Banach space $Z$ is said to have the
{\sf bounded compact approximation property (BCAP)} if there
exists a constant $M \in (0, +\infty)$ such that given any $\varepsilon > 0$
and any finite set $F \subset Z$, there exists an operator
$T \in \mathcal{K}(Z)$ such that 
\begin{equation}\label{M}
\|I - T\| \le M  \quad\mbox{and}\quad   \|y - Ty\| < \varepsilon , \ \ \forall y \in F .
\end{equation}
Here $I$ is the identity map from $Z$ to itself. 

We say that $Z$ has the
{\sf dual compact approximation property (DCAP)} if there
exists a constant $M^* \in (0, +\infty)$ such that given any $\varepsilon > 0$
and any finite set $G \subset Z^*$, there exists an operator
$T \in \mathcal{K}(Z)$ such that 
\begin{equation}\label{M*}
\|I - T\| \le M^* \quad\mbox{and}\quad \|z - T^*z\| < \varepsilon , \ \ \forall z \in G .
\end{equation}
The greatest lower bound of the constants $M$  (constants $M^*$) for which \eqref{M}  (\eqref{M*}, respectively) holds will be denoted by
$M(Z)$ (by $M^*(Z)$). 
\end{definition}
It is easy to see that 
$$
Z \mbox{ has the DCAP } \Longrightarrow \ Z^* \mbox{ has the BCAP and } M(Z^*) \le M^*(Z),
$$
and that if $Z$ is reflexive, then
$$
Z \mbox{ has the DCAP } \Longleftrightarrow \ Z^* \mbox{ has the BCAP and } M(Z^*) = M^*(Z).
$$
It was proved in \cite{S93}\footnote{The terminology used in \cite{S93} is slightly different from, and perhaps more appropriate than, ours. 
What we call the DCAP of $Z$ is called the $\ast$--BCAP of $Z^*$ there.} 
that if $Z$ has the DCAP then it also has the BCAP, while there exists a (non-reflexive) Banach space that has the BCAP
but not the DCAP. Hence one has the following if $Z$ is reflexive
$$
Z \mbox{ has the BCAP } \Longleftrightarrow \ Z \mbox{ has the DCAP } \Longleftrightarrow \ Z^* \mbox{ has the BCAP}.
$$
It does not seem to be known whether there exists a (non-reflexive) Banach space $Z$ such that $Z^*$ has the BCAP, but $Z$ does not have the DCAP,
or a  Banach space $X$ such that $M(X^*) \not= M^*(X)$.

Although we will apply the results of this Section only to reflexive spaces, we consider here the general (non-reflexive) case.

A comprehensive study of various
approximation properties can be found in \cite{LT} (see also \cite{ES05} for examples of function spaces that have the BCAP).

If $Y$ has the BCAP, then 
\begin{equation}\label{esschi}
\|A\|_{\mathrm{e}} \le M(Y) \|A\|_\chi\, , \ \ \ \forall A \in \mathcal{B}(X, Y)
\end{equation}
(see \cite{LS}) and hence 
$$
\|A\|_{\mathrm{e}} \le 2 M(Y) \|A\|_m\, , \ \ \ \forall A \in \mathcal{B}(X, Y)
$$
(see \eqref{nontrivial}). 

\begin{theorem}\label{dcap} 
If $X$ has the DCAP, then
\begin{equation}\label{essm}
\|A\|_{\mathrm{e}} \le  M^*(X) \|A\|_m\, , \ \ \ \forall A \in \mathcal{B}(X, Y) .
\end{equation}
\end{theorem}
\begin{proof}
Take any $\varepsilon > 0$. According to the definition of $\|A\|_m$, there exists a subspace $M$ of $X$ having finite codimension and such  that 
\begin{equation}\label{m}
\|Ax\| \le (\|A\|_m + \varepsilon) \|x\| , \ \ \ \forall x \in M .
\end{equation}
Let $Q : X \to M$ be a bounded projection onto $M$. Then $I - Q$ is a finite rank operator. Since $(I - Q)^* \in \mathcal{K}(X^*)$, there exist
$z_1, \dots, z_n \in X^*$ such that
\begin{equation}\label{k}
\min_{k = 1, \dots, n} \|(I - Q)^*z - z_k\| < \varepsilon
\end{equation}
for every $z \in X^*$ with $\|z\| \le 1$. Since $X$ has the DCAP,  there exists 
$T \in \mathcal{K}(X)$ such that $\|I - T\| \le M^*(X)$ and
$$
\|z_k - T^*z_k\| < \varepsilon , \ \ \  k = 1, \dots, n.
$$
Take any $z \in X^*$ with $\|z\| \le 1$ and choose $k$ for which the minimum in \eqref{k} is achieved. Then
\begin{eqnarray*}
\|(I -T)^*(I - Q)^*z\| \le  \left\|(I -T)^*\left((I - Q)^*z - z_k\right)\right\| +  \|(I -T)^* z_k\| \\
< \|(I -T)^*\|\varepsilon + \varepsilon = \|I - T\|\varepsilon + \varepsilon \le (M^*(X) + 1) \varepsilon .
\end{eqnarray*}
Hence
$$
\|(I - Q)(I - T)\| = \|(I -T)^*(I - Q)^*\| < (M^*(X) + 1) \varepsilon .
$$
Then using \eqref{m}, one gets
\begin{eqnarray*}
&& \|AQ(I - T)x\|  
\le (\|A\|_m + \varepsilon)\|Q(I - T)x\|  \\
&& \le  (\|A\|_m + \varepsilon)\Big(\|(I - T)x\| +  \|(I - Q)(I - T)x\|\Big)  \\
&& \le  (\|A\|_m + \varepsilon)\left(M^*(X) +  (M^*(X) + 1) \varepsilon\right) 
\end{eqnarray*}
for every $x \in X$ with $\|x\| \le 1$. Since
$$
A - AQ(I - T) = A(I - Q) + AQT \in \mathcal{K}(X, Y) ,
$$
the above implies
$$
\|A\|_{\mathrm{e}} \le  (\|A\|_m + \varepsilon)\left(M^*(X) +  (M^*(X) + 1) \varepsilon\right) 
$$
for any $\varepsilon > 0$. Hence
$$
\|A\|_{\mathrm{e}} \le \|A\|_m M^*(X) .
$$
\end{proof}

Estimate \eqref{esschi} is sharp in the following sense: if 
$$
\|A\|_{\mathrm{e}} <  M \|A\|_\chi
$$
holds for every $A \in \mathcal{B}(X, Y)\setminus \mathcal{K}(X,Y)$ and every Banach space $X$, then
$Y$ has the BCAP and $M(Y) \le M$ (see \cite[Theorem 2.3]{AT87}). The following theorem shows that 
a similar result holds for \eqref{essm} (see also \cite[Theorems 1.2, 2.2, and Proposition 2.1]{Ty} and Remark  \ref{equal} above).

\begin{theorem}\label{dcapAT} 
If 
\begin{equation}\label{esslambda}
\|A\|_{\mathrm{e}} <  M \|A\|_m
\end{equation}
holds for every $A \in \mathcal{B}(X, Y)\setminus \mathcal{K}(X,Y)$ and every Banach space $Y$, then
$X$ has the DCAP and $M^*(X) \le M$.
\end{theorem}
\begin{proof}
The proof is similar to that of \cite[Theorem 2.3]{AT87}. 
Suppose $X$ does not have the DCAP or $M^*(X) > M$. Then there exists a finite set $G \subset Y^*$ and 
an $\varepsilon > 0$ such that if \eqref{M*} holds with $M^* = M$ for an operator $T \in \mathcal{B}(X)$, then $T \not\in  \mathcal{K}(X)$.

Let $G = \{g_1, \dots, g_N\}$ and let $Y$ be the vector space $X$ equipped with the norm
$$
\|x\|_Y := \frac{\varepsilon}{M}\, \|x\|_X + \sum_{n = 1}^N |g_n(x)| .
$$
It is clear that
\begin{equation}\label{norms}
\frac{\varepsilon}{M}\, \|x\|_X \le \|x\|_Y \le \left(\frac{\varepsilon}{M} + \sum_{n = 1}^N \|g_n\|_{X^*}\right) \|x\|_X 
\quad\mbox{for all}\quad x \in X.
\end{equation}
Hence the norm $\|\cdot\|_Y$ is equivalent to $\|\cdot\|_X$, and $Y$ is a Banach space isomorphic, but not isometric, to $X$. 

Let $A  \in \mathcal{B}(X, Y)$ be the identity operator. Suppose $T \in \mathcal{B}(X, Y) = \mathcal{B}(X)$ is such that 
$\|A - T\|_{\mathcal{B}(X, Y)} < \varepsilon$. Then
\begin{eqnarray*}
&& \|g_n - T^*g_n\|_{X^*} = \sup_{\|x\|_X = 1} |g_n(x) - T^*g_n(x)| = \sup_{\|x\|_X = 1} |g_n(x) - g_n(T x)| \\
&& =  \sup_{\|x\|_X = 1} |g_n\left(x - T x\right)| \le\ \sup_{\|x\|_X = 1} \|x -Tx\|_Y = \|A - T\|_{\mathcal{B}(X, Y)} < \varepsilon 
\end{eqnarray*}
for all $n = 1, \dots, N$. It follows from \eqref{norms} that
\begin{eqnarray*}
\|x - Tx\|_X \le \frac{M}{\varepsilon}\, \|x - Tx\|_Y \le \frac{M}{\varepsilon}\, \|A - T\|_{\mathcal{B}(X, Y)} \|x\|_X \\
\le M \|x\|_X \quad\mbox{for all}\quad x \in X .
\end{eqnarray*}
Hence $\|I - T\|_{\mathcal{B}(X)} \le M$, and $T$ satisfies \eqref{M*} with $M^* = M$. Then $T$ is not compact
according to our assumption. This means that $\|A\|_{\mathrm{e}} \ge \varepsilon$.

On the other hand, let
\begin{eqnarray*}
L := \bigcap_{n = 1}^N \mathrm{Ker}\, g_n 
= \{x \in X : \ g_1(x) = \cdots = g_N(x) = 0\} .
\end{eqnarray*}
Then the codimension of $L$ is less than or equal to $N$, and
$$
\|Ax\|_Y = \|x\|_Y = \frac{\varepsilon}{M}\, \|x\|_X
 \quad\mbox{for all}\quad x \in L .
$$
Hence $\|A\|_m \le \frac{\varepsilon}{M}$, and $\|A\|_{\mathrm{e}} \ge M \|A\|_m$. This contradicts \eqref{esslambda} and proves
that $M^*(X)$ has to be less than or equal to $M$.
\end{proof}

\section{Approximation properties of Hardy spaces}\label{bcaph}

\begin{theorem}\label{apprHp} 
The Hardy space $H^p = H^p(\mathbb{T})$, $1 < p < \infty$ has the bounded compact approximation and the
dual compact approximation properties with
$$
M(H^p), M^*(H^p) \le 2^{\left|1 - \frac2p\right|}\, .
$$
\end{theorem}
\begin{proof}
Let
\begin{eqnarray*}
K_n\left(e^{i\theta}\right) := \frac1{2\pi}\sum_{k = -n}^n \left(1 - \frac{|k|}{n + 1}\right) e^{i k \theta} = 
\frac1{2\pi(n + 1)}\left(\frac{\sin\frac{(n + 1)\theta}{2}}{\sin\frac{\theta}2}\right)^2 ,
\\ \theta \in [-\pi, \pi] , \ \ n = 0, 1, 2, \dots
\end{eqnarray*}
be the $n$-th Fej\'er kernel, and let
$$
\left(\mathbf{ K}_n f\right)\left(e^{i\vartheta}\right) := \left(K_n\ast f\right)\left(e^{i\vartheta}\right) =
\int_{-\pi}^\pi K_n\left(e^{i\vartheta - i\theta}\right) f\left(e^{i\theta}\right)\, d\theta , \ \ \ \vartheta \in [-\pi, \pi] ,
$$
where $f \in L^1(\mathbb{T})$. It is well known that $\|K_n\|_{L^1(\mathbb{T})} = 1$,
\begin{equation}\label{Fej}
\left(\mathbf{ K}_n f\right)\left(e^{i\vartheta}\right) = \sum_{k = -n}^n \widehat{f}(k) \left(1 - \frac{|k|}{n + 1}\right) e^{i k \theta} ,
\end{equation}
where $\widehat{f}(k)$ is the $k$-th Fourier coefficient of $f$, 
$\|\mathbf{ K}_n\|_{L^p \to L^p} = 1$ for $1 \le p \le \infty$, $\mathbf{ K}_n$ converge strongly to the identity operator on $L^p(\mathbb{T})$, $1 \le p < \infty$ as $n \to \infty$,
and $\mathbf{ K}_n$ map $H^p(\mathbb{T})$ into itself (see, e.g., \cite[Ch. 2]{H62}, \cite[Ch. I, $\S$2]{K68}). It follows from \eqref{Fej} and Parseval's theorem that
$\|I - \mathbf{ K}_n\|_{L^2 \to L^2} = 1$. Since $\|I - \mathbf{ K}_n\|_{L^p \to L^p} \le 1 +\|\mathbf{ K}_n\|_{L^p \to L^p} = 2$, 
the Riesz-Thorin interpolation theorem (see, e.g., \cite[Theorem 9.3.3]{G07}) applied to $L^2(\mathbb{T})$ and $L^\infty(\mathbb{T})$
implies
$$
\|I - \mathbf{ K}_n\|_{L^p \to L^p} \le 2^{1 - \frac2p}\, , \ \ \ 2 \le p \le \infty .
$$
Similarly, interpolating between $L^2(\mathbb{T})$ and $L^1(\mathbb{T})$, one gets
$$
\|I - \mathbf{ K}_n\|_{L^p \to L^p} \le 2^{\frac2p - 1}\, , \ \ \ 1 \le p \le 2 .
$$
The above inequalities imply that
$$
\|I - \mathbf{ K}_n\|_{H^p \to H^p} \le  2^{\left|1 - \frac2p\right|}\, , \ \ \ 1 \le p \le \infty .
$$
It is easy to see that the adjoint to $\mathbf{ K}_n : H^p(\mathbb{T}) \to H^p(\mathbb{T})$, $1 < p < \infty$ operator
can be identified with $\mathbf{ K}_n : H^{p'}(\mathbb{T}) \to H^{p'}(\mathbb{T})$, $p' = \frac{p}{p - 1}$ (see \cite[$\S$7.2]{D70}).
Hence the conditions in Definition \ref{def1} are satisfied for $T = \mathbf{ K}_n$ with a sufficiently large $n$.
\end{proof}
If $1 < p < \infty$, the Hardy space $H^p(\mathbb{T})$ is reflexive. Although $(H^p(\mathbb{T}))^*$ is isomorphic to
$H^{p'}(\mathbb{T})$, these two spaces are not isometrically isomorphic, and it is not clear whether or not $M^*(H^p) = M(H^{p'})$.
Unfortunately, the exact values of $M(H^p)$, $M^*(H^p)$ do not seem to be known. Since $(L^p(\mathbb{T}))^*$ is isometrically isomorphic to
$L^{p'}(\mathbb{T})$, one has $M^*(L^p) = M(L^{p'})$. It is known that
$$
M(L^p) = C_p:= 
 \max_{0 < \alpha < 1} \left(\alpha^{p - 1} + (1 - \alpha)^{p - 1}\right)^{\frac1p} \left(\alpha^{\frac1{p - 1}} + (1 - \alpha)^{\frac1{p - 1}}\right)^{1 - \frac1p} 
$$
(see \cite{SS}), $C_{p'} = C_p$, and 
$$
1 \le C_p \le 2^{\left|1 - \frac2p\right|} ,
$$
where the left inequality is strict unless $p = 2$, while the right one is strict unless $p = 1, 2$ or $\infty$ (see \cite{F90} and \cite{M09}).

\section{An upper estimate for the essential norm of a Toeplitz operator}\label{2pp'}

\begin{theorem}\label{upperest} 
Let $a \in (C + H^\infty)(\mathbb{T})$. Then the following holds for the Toeplitz operator $T(a) : H^p(\mathbb{T}) \to H^p(\mathbb{T})$, $1 < p < \infty$,
\begin{equation}\label{upper}
\|T(a)\|_m = \|a\|_{L^\infty} , \ \ \ \|T(a)\|_{\mathrm{e}}  \le 2^{\left|1 - \frac2p\right|}\|a\|_{L^\infty}  .
\end{equation}
\end{theorem}
\begin{proof}
It is sufficient to prove the equality in \eqref{upper} as the inequality then follows from Theorems \ref{dcap} and \ref{apprHp}. 
Since $\|T(b)\| \le \|P\| \|b\|_{L^\infty}$ for any $b \in L^\infty(\mathbb{T})$ and functions of the form $a = \mathbf{e}_{-n}h$, $h \in H^\infty(\mathbb{T})$,
$n \in \mathbb{N}$ are dense in $(C + H^\infty)(\mathbb{T})$ (see, e.g, \cite[Ch. IX, Theorem 2.2]{G81}), it is sufficient to prove \eqref{upper} for such a function.
Let $H_n^p(\mathbb{T})$ be the subspace of $H^p(\mathbb{T})$ consisting of all functions with the first $n$ Fourier coefficients equal to $0$. Then
$H_n^p(\mathbb{T})$ has codimension $n$ and
$$
\|T(a) f\|_{H^p} = \|af\|_{H^p} \le \|a\|_{L^\infty} \|f\|_{H^p} , \ \ \ \forall f \in H_n^p(\mathbb{T}) .
$$
Hence $\|T(a)\|_m \le \|a\|_{L^\infty}$. Since $\|T(a)\|_m$ is greater than or equal to the essential spectral radius of $T(a)$ (see \cite[$\S$6]{LS}) and the latter
is greater than or equal to $\|a\|_{L^\infty}$ (see, e.g., \cite[Theorem 2.30]{BS06}), one has the opposite inequality $\|T(a)\|_m \ge \|a\|_{L^\infty}$.
\end{proof}

\section{The essential norm of the backward shift operator}\label{model}

\begin{theorem}\label{lowerest} 
The following equalities hold for the Toeplitz operator $T(\mathbf{e}_{-1}) : H^p(\mathbb{T}) \to H^p(\mathbb{T})$, $1 < p < \infty$,
\begin{equation}\label{lower}
\|T(\mathbf{e}_{-1})\|_\chi =  \|T(\mathbf{e}_{-1})\|_{\mathrm{e}} = \|T(\mathbf{e}_{-1})\|_{H^p \to H^p}\, .
\end{equation}
\end{theorem}
\begin{proof}
Since
$$
\|T(\mathbf{e}_{-1})\|_\chi \le  \|T(\mathbf{e}_{-1})\|_{\mathrm{e}} \le \|T(\mathbf{e}_{-1})\|_{H^p \to H^p}
$$
(see \eqref{trivial}), it is sufficient to prove that $\|T(\mathbf{e}_{-1})\|_\chi \ge \|T(\mathbf{e}_{-1})\|_{H^p \to H^p} =: C_p$.
For any $\varepsilon > 0$, there exists $q \in H^p(\mathbb{T})$,
$$
q\left(e^{i\theta}\right) = \sum_{k = 0}^\infty c_k e^{i k\theta} = c_0 + \sum_{k = 1}^\infty c_k e^{i k\theta} =: c_0 + q_0\left(e^{i\theta}\right) , 
\ \ \ \theta \in [-\pi, \pi] , 
$$
such that $\|q\|_{H^p} = 1$ and  
$$
\|q_0\|_{H^p} = \|\mathbf{e}_{-1}q_0\|_{H^p} = \|T(\mathbf{e}_{-1})q\|_{H^p} \ge C_p - \varepsilon .
$$
Since $\mathbf{e}_N$, $N \in \mathbb{N}$ is an inner function and $\mathbf{e}_N(0) = 0$, one has
$\|f\circ\mathbf{e}_N\|_{H^p} = \|f\|_{H^p}$ for any $f \in H^p(\mathbb{T})$
(see (\cite{N68}, \cite[Section 1.3]{Shap}, and \cite[Theorem 5.5]{DG}).
Hence $\|q\circ\mathbf{e}_N\|_{H^p} = 1$ and 
$$
\|T(\mathbf{e}_{-1})(q\circ\mathbf{e}_N)\|_{H^p} = \|\mathbf{e}_{-1}(q_0\circ\mathbf{e}_N)\|_{H^p} 
 = \|q_0\circ\mathbf{e}_N\|_{H^p} \ge C_p - \varepsilon .
$$
Take any finite set $\{\varphi_1, \dots, \varphi_m\} \subset H^p(\mathbb{T})$ and choose polynomials
$$
\psi_j(z) := \sum_{k = 0}^{n_j} \psi_{j, k} z^k , \ \ \ z \in \mathbb{C} , \ \ j = 1, \dots, m
$$
such that $\|\varphi_j - \psi_j\|_{H^p} \le \varepsilon$. Then choose $N \in \mathbb{N}$ such that
\begin{equation}\label{N}
N > \max\{n_1, \dots, n_m\} + 1
\end{equation}
and set $h := \|f\|_{H^p}^{1 - p} |f|^{p - 2} \overline{f}$, where $f = q_0\circ\mathbf{e}_N$. A standard calculation gives
$\|h\|_{L^{p'}} = 1$ and 
$$
\int_{-\pi}^\pi (q_0\circ\mathbf{e}_N)\left(e^{i\theta}\right) h\left(e^{i\theta}\right)\, d\theta = \|q_0\circ\mathbf{e}_N\|_{H^p} .
$$
The Fourier series of $h$ has the form
$$
\sum_{k \in \mathbb{Z}} h_k e^{ik N\theta} , \ \ \ h_k \in \mathbb{C} .
$$
It follows from \eqref{N} that
$$
\{k N | \ k \in \mathbb{Z}\} \cap \{1, \dots , n_j + 1\} = \emptyset , \ \ \ j = 1, \dots, m .
$$
Hence
$$
\int_{-\pi}^\pi (\mathbf{e}_1 \psi_j)\left(e^{i\theta}\right) h\left(e^{i\theta}\right)\, d\theta = 0 .
$$
So,
$$
\int_{-\pi}^\pi (q_0\circ\mathbf{e}_N - \mathbf{e}_1 \psi_j)\left(e^{i\theta}\right) h\left(e^{i\theta}\right)\, d\theta = \|q_0\circ\mathbf{e}_N\|_{H^p} ,
\ \ \ j = 1, \dots, m .
$$
On the other hand, H\"older's inequality implies 
$$
\left|\int_{-\pi}^\pi (q_0\circ\mathbf{e}_N - \mathbf{e}_1 \psi_j)\left(e^{i\theta}\right) h\left(e^{i\theta}\right)\, d\theta\right| \le 
\|q_0\circ\mathbf{e}_N - \mathbf{e}_1 \psi_j\|_{H^p}, 
$$
since $\|h\|_{L^{p'}} = 1$. Hence
$$
\|q_0\circ\mathbf{e}_N - \mathbf{e}_1 \psi_j\|_{H^p} \ge \|q_0\circ\mathbf{e}_N\|_{H^p}
$$
and
\begin{eqnarray*}
\|T(\mathbf{e}_{-1})(q\circ\mathbf{e}_N) - \varphi_j\|_{H^p} = \|\mathbf{e}_{-1}(q_0\circ\mathbf{e}_N) - \varphi_j\|_{H^p}
= \|q_0\circ\mathbf{e}_N - \mathbf{e}_1\varphi_j\|_{H^p} \\
\ge \|q_0\circ\mathbf{e}_N - \mathbf{e}_1 \psi_j\|_{H^p} - \varepsilon \ge \|q_0\circ\mathbf{e}_N\|_{H^p} - \varepsilon \ge C_p - 2\varepsilon ,
 \ \ j = 1, \dots, m .
\end{eqnarray*}
So, for every finite set $\{\varphi_1, \dots, \varphi_m\} \subset H^p(\mathbb{T})$, there exist an element of the image of the unit ball
$T(\mathbf{e}_{-1})\left(B_{H^p}\right)$ that lies at a distance at least $C_p - 2\varepsilon$ from every element of $\{\varphi_1, \dots, \varphi_m\}$.
This means that $T(\mathbf{e}_{-1})\left(B_{H^p}\right)$ cannot 
be covered by a finite family of open balls of radius $C_p - 2\varepsilon$. Hence
$$
\|T(\mathbf{e}_{-1})\|_\chi \ge C_p - 2\varepsilon = \|T(\mathbf{e}_{-1})\|_{H^p \to H^p} - 2 \varepsilon , \ \ \ \forall \varepsilon > 0 ,
$$
i.e. $\|T(\mathbf{e}_{-1})\|_\chi \ge \|T(\mathbf{e}_{-1})\|_{H^p \to H^p}$.
\end{proof}

It is known that $\|T(\mathbf{e}_{-1})\|_{H^p \to H^p} > 1$ for $p \not= 2$ (see \cite[$\S$ 7]{BKS}). So, it follows from Theorem \ref{lowerest} that
\begin{equation}\label{greater}
\|T(\mathbf{e}_{-1})\|_{\mathrm{e}} > 1 = \|\mathbf{e}_{-1}\|_{L^\infty} , \quad p \not=2 .
\end{equation}
The exact value of $\|T(\mathbf{e}_{-1})\|_{H^p \to H^p}$ does not seem to be known (see \cite{F17}), but it follows from the proof of Theorem \ref{apprHp} that
$$
\|T(\mathbf{e}_{-1})\|_{H^p \to H^p} = \|\mathbf{e}_{-1} (I - \mathbf{ K}_0)\|_{H^p \to H^p} = \|I - \mathbf{ K}_0\|_{H^p \to H^p} \le
2^{\left|1 - \frac2{p}\right|}
$$
(see \cite[7.8]{BKS} and \cite{F17}).

\begin{remark}
{\rm The proof of Theorem \ref{lowerest} implies 
$$
\|T(\mathbf{e}_{-n})\|_{H^p \to H^p} \ge \|T(\mathbf{e}_{-1})\|_{H^p \to H^p} , \ \ \ \forall n \in \mathbb{N} .
$$
Indeed, 
$\|q\circ\mathbf{e}_N\|_{H^p} = 1$ and for any $N \ge n$ one has
\begin{eqnarray*}
\|T(\mathbf{e}_{-1})\|_{H^p}  - \varepsilon &\le& \|q_0\circ\mathbf{e}_N\|_{H^p} = \|\mathbf{e}_{-n}(q_0\circ\mathbf{e}_N)\|_{H^p} = 
\|T(\mathbf{e}_{-n})(q\circ\mathbf{e}_N)\|_{H^p} \\
&\le& \|T(\mathbf{e}_{-n})\|_{H^p}  .
\end{eqnarray*}
The same argument proves also the  following inequalities for the Fej\'er means:
$$
\|I - \mathbf{ K}_n\|_{H^p \to H^p} \ge \|I - \mathbf{ K}_0\|_{H^p \to H^p} , \ \  \|I - \mathbf{ K}_n\|_{L^p \to L^p} \ge \|I - \mathbf{ K}_0\|_{L^p \to L^p} ,
\ \ \ \forall n \in \mathbb{N} .
$$}
\end{remark}

\section{Toeplitz operators on abstract Hardy spaces}\label{concl}
Some of the above results can be extended to more general Banach function spaces.

Let $L^0_+$ be the
set of measurable functions whose values lie in $[0,\infty]$. 
Following \cite[Chap.~1, Definition~1.1]{BS88}, a mapping 
$\rho: L_+^0\to [0,\infty]$ is called a Banach function norm
if, for all functions $f,g, f_n\in L_+^0$ with $n\in\mathbb{N}$, and for all
constants $a\ge 0$, the
following  properties hold:
\begin{eqnarray*}
{\rm (A1)} & &
\rho(f)=0  \Leftrightarrow  f=0\ \mbox{a.e.},
\
\rho(af)=a\rho(f),
\
\rho(f+g) \le \rho(f)+\rho(g),\\
{\rm (A2)} & &0\le g \le f \ \mbox{a.e.} \ \Rightarrow \ 
\rho(g) \le \rho(f)
\quad\mbox{(the lattice property)},\\
{\rm (A3)} & &0\le f_n \uparrow f \ \mbox{a.e.} \ \Rightarrow \
       \rho(f_n) \uparrow \rho(f)\quad\mbox{(the Fatou property)},\\
{\rm (A4)} & & \rho(1) <\infty,\\
{\rm (A5)} & &\int_{-\pi}^\pi f(e^{i\theta})\,d\theta \le C\rho(f)
\end{eqnarray*}
with {a constant} $C \in (0,\infty)$ that may depend on  
$\rho$,  but is independent of $f$. When functions differing only on 
a set of measure  zero are identified, the set $X$ of all functions 
$f\in L^0$ for  which  $\rho(|f|)<\infty$ is called a Banach function
space. For each $f\in X$, the norm of $f$ is defined by
$\|f\|_X :=\rho(|f|)$.
The set $X$ equipped with the natural linear space operations and with 
this norm becomes a Banach space (see 
\cite[Chap.~1, Theorems~1.4 and~1.6]{BS88}). 
If $\rho$ is a Banach function norm, its associate norm 
$\rho'$ is defined on $L_+^0$ by
\[
\rho'(g):=\sup\left\{
\int_{-\pi}^\pi f(e^{i\theta}) (e^{i\theta})\,d\theta \ : \ 
f\in L_+^0, \ \rho(f) \le 1
\right\}, \ g\in L_+^0.
\]
It is a Banach function norm itself \cite[Chap.~1, Theorem~2.2]{BS88}.
The Banach function space $X'$ determined by the Banach function norm
$\rho'$ is called the associate space (K\"othe dual) of $X$. 
The associate space $X'$ can be viewed as a subspace of the (Banach) 
dual space $X^*$.

Let $X = X(\mathbb{T})$ be a Banach function space and let 
$$
H[X] := \left\{f \in X : \ \widehat{f}(k) = 0 \ \mbox { for all } \ k < 0\right\} 
$$
be the corresponding Hardy space. If the Riesz projection $P$ is bounded on $X$, then it projects $X$ onto
$H[X]$, and one can define the Toeplitz operator $T(a)$ similarly to \eqref{T}:
\begin{equation}\label{TX}
T(a)  = PaI : H[X] \to H[X] ,  \quad a \in L^\infty(\mathbb{T}) .
\end{equation}
\begin{lemma}\label{HWSl} 
Let $X$ be a separable Banach function space on which the Riesz projection is bounded, and
let $a \in L^\infty(\mathbb{T})$. If the Toeplitz operator $T(a) : H[X] \to H[X]$ is
Fredholm, then $\frac1a \in L^\infty(\mathbb{T})$.
\end{lemma}
\begin{proof}
Let $U := \mathbf{e}_1 I$. It is easy to see that $U^n$, $n \in \mathbb{N}$ converges weakly to zero on $X$. Indeed, since $X$ is separable, its 
Banach space dual is canonically isometrically isomorphic to its associate space: $X^* = X'$ 
(see \cite[Ch. 1, Corollaries 4.3 and 5.6]{BS88}). Hence it is sufficient to show that
\begin{equation}\label{RL}
 \int_{-\pi}^\pi e^{i n\theta} g\left(e^{i\theta}\right) h\left(e^{i\theta}\right)\, d\theta \to 0 \ \mbox{ as } \ n \to \infty
\end{equation}
for all $g \in X$ and $h \in X'$. Since $gh \in L^1(\mathbb{T})$ (see \cite[Ch. 1, Theorem 2.4]{BS88}), \eqref{RL} follows from
the Riemann-Lebesgue lemma (see, e.g., \cite[Ch. I, Theorem 2.8]{K68}).

Since $X$ is separable, the set of trigonometric polynomials is dense in $X$ (see, e.g., \cite[Corollary 2.2]{KS19}).

Using the above facts, one can show exactly as in the proof of \cite[Theorem 2.30(a)]{BS06} that there exist $\delta > 0$ such that
$\|a g\|_X \ge \delta \|g\|_X$ for all $g \in X(\mathbb{T})$. Suppose $|a| \le \frac{\delta}2$ on a set $E \subseteq \mathbb{T}$ of positive measure. Then for
every function $g \in X(\mathbb{T})$ supported in $E$ one has $|ag| \le \frac{\delta}2\, |g|$ a.e. on $\mathbb{T}$, and hence
$$
\frac{\delta}2\, \|g\|_X \ge \|a g\|_X \ge \delta \|g\|_X ,
$$
which is a contradiction. So, $|a| > \frac{\delta}2$ a.e. on $\mathbb{T}$, and $\frac1a \in L^\infty(\mathbb{T})$.
\end{proof}
If $X$ is reflexive, then the above result follows from \cite[Theorems 6.11 and 7.3]{K03}.
\begin{corollary}\label{HWSc} 
Let $X$ be a separable Banach function space on which the Riesz projection is bounded, and
let $a \in L^\infty(\mathbb{T})$. Then \eqref{HWS} and \eqref{essnr} hold the Toeplitz operator $T(a) : H[X] \to H[X]$.
\end{corollary}
\begin{proof}
It is sufficient to apply Lemma \ref{HWSl} to $a - \lambda$ with $\lambda \not\in \text{\rm Spec}_{\rm e}(T(a))$.
\end{proof}

Let $w$ be a measurable function such that $0 < w < \infty$ a.e. on $\mathbb{T}$ and
\begin{equation}\label{weight}
w \in X , \quad \frac{1}{w} \in X' .
\end{equation}
Then the weighted space $X(w) = X(\mathbb{T}, w)$ consisting of all measurable functions $g$ such that
$$
\|g\|_{X(w)} := \|wg\|_{X} < \infty
$$
is a Banach function space (see \cite[Lemma 2.4(b)]{KS14}). 
\begin{lemma}\label{iso} 
If $w \ge 0$ satisfies \eqref{weight}, then $H[X(w)]$ is isometrically isomorphic to $H[X]$. 
\end{lemma}
\begin{proof}
Let $\mathbb{D}$ be the unit disc: $\mathbb{D} :=\left\{z \in \mathbb{C} : \ |z| < 1\right\}$. A function $F$ analytic in $\mathbb{D}$ is said to belong to the Hardy
space $H^p(\mathbb{D})$, $0<p\le\infty$, if the integral mean
\begin{align*}
&
M_p(r,F)=\left(\frac{1}{2\pi}\int_{-\pi}^\pi |F(re^{i\theta})|^p\,d\theta\right)^{1/p},
\quad
0<p<\infty,
\\
&
M_\infty(r,F)=\max_{-\pi\le\theta\le\pi}|F(re^{i\theta})|,
\end{align*}
remains bounded as $r\to 1$. If $F\in H^p(\mathbb{D})$, $0<p\le\infty$, then
the nontangential limit $F(e^{i\theta})$ exists almost everywhere on $\mathbb{T}$ and
$F \in L^p(\mathbb{T})$ (see, e.g., \cite[Theorem 2.2]{D70}). If $1 \le p \le \infty$, then
$F \in H^p(\mathbb{T})$ (see, e.g., \cite[Theorem 3.4]{D70}). 

It follows from \eqref{weight} and Axiom (A5) that $w \in L^1(\mathbb{T})$, $\frac{1}{w} \in L^1(\mathbb{T})$. Then $\log w  \in L^1(\mathbb{T})$.
Consider the outer function
$$
W(z) := \exp\left(\frac1{2\pi}\int_{-\pi}^\pi\frac{e^{it} + z}{e^{it} - z}\, \log w(e^{it})\, dt\right), \quad z \in \mathbb{D}
$$
(see \cite[Ch. 5]{H62}). It belongs to $H^1(\mathbb{D})$ and $|W| = w$ a.e. on $\mathbb{R}$.

It follows from the definition of $X(w)$ that
\begin{equation}\label{Weight}
\|Wf\|_X = \|wf\|_X = \|f\|_{X(w)} \quad\mbox{for all}\quad f \in H[X(w)] .
\end{equation}
Since $X(w)$ is a Banach function space, Axiom (A5) implies that $X(w) \subseteq L^1(\mathbb{T})$ and $H[X(w)] \subseteq H^1(\mathbb{T})$.
Take any $f \in H[X(w)] $. Let $F \in H^1(\mathbb{D})$ be its analytic extensions to the unit disk $\mathbb{D}$
by means of the Poisson integral (see the proof of \cite[Theorem 3.4]{D70}).  
Since $W, F \subseteq H^1(\mathbb{D})$,  H\"older's inequality implies that
$WF \in H^{1/2}(\mathbb{D})$. It follows from \eqref{Weight} and Axiom (A5) that $Wf \in X \subseteq L^1(\mathbb{T})$.
Hence $WF \in H^1(\mathbb{D})$ (see \cite[Theorem 2.11]{D70}). So, $Wf \in H^1(\mathbb{T})\cap X = H[X]$. This proves that
the mapping $f \mapsto Wf$ is an isometric isomorphism of $H[X(w)]$ into $H[X]$. 

Repeating the above argument, one gets that
the mapping $g \mapsto \frac1W\, g$ is an isometric isomorphism of $H[X]$ into $H[X(w)]$. Hence $H[X(w)]$ and $H[X]$ are isometrically isomorphic.
\end{proof}

We say that $X = X(\mathbb{T})$ is translation-invariant if 
$\|\tau_\vartheta f\|_X = \|f\|_X$ for every $\vartheta \in [-\pi, \pi]$ and every $f \in X$, where
$$
(\tau_\vartheta f)\left(e^{i\theta}\right) := f\left(e^{i\theta - i\vartheta}\right) , \ \ \ \theta \in [-\pi, \pi] .
$$
\begin{lemma}\label{transll} 
Let $X$ be a separable translation-invariant Banach function space and let $w \ge 0$ satisfy \eqref{weight}. 
Then $M(H[X(w)]) \le 2$. Additionally, if $X$ is reflexive, then $M^*(H[X(w)]) \le 2$.
\end{lemma}
\begin{proof}
First, we consider the case $w \equiv  1$. Since $X$ is separable, smooth functions are dense in $X$ (see \cite{KS14}). 
So, $\lim_{\vartheta \to 0}\|\tau_\vartheta f - f\|_X = 0$ for every $f \in X$. One then has,
as in the proof of Theorem \ref{apprHp}, that $\|\mathbf{ K}_n\|_{X \to X} = 1$, $\|I - \mathbf{ K}_n\|_{X \to X} \le 1 + \|\mathbf{ K}_n\|_{X \to X} = 2$,
$\mathbf{ K}_n$ converge strongly to the identity operator on $X$ as $n \to \infty$, and
$\mathbf{ K}_n$ maps $H[X]$ to $H[X]$ (see \cite[Ch. I, $\S$2]{K68} and \cite[Ch. 2]{H62}). Hence $M(H[X]) \le 2$.

If $X$ is reflexive, then $X^* = X'$ is also separable (see \cite[Ch. 1, Corollaries 4.4 and 5.6]{BS88}) and translation-invariant (see, e.g., \cite[Lemma 2.1]{KS19F}).
Then it follows from the above that the adjoint operators $\mathbf{ K}_n^* = \mathbf{ K}_n : X' \to X'$ converge strongly to the identity operator as $n \to \infty$.
Applying \eqref{adj} to $A = I - \mathbf{ K}_n$, $X_0 = Y_0 = H[X]$, one concludes that he adjoint operators $\mathbf{ K}_n^* : (H[X])^* \to (H[X])^*$ 
also converge strongly to the identity operator as $n \to \infty$. Hence $M^*(H[X]) \le 2$. This concludes the proof for $w \equiv  1$.
The case of a general weight $w$ now follows from Lemma \ref{iso}.
\end{proof}

\begin{theorem}\label{transl} 
Let $X$ be a reflexive translation invariant Banach function and let $w \ge 0$ satisfy \eqref{weight}.
If the Riesz projection is bounded on $X(w)$ and
$a \in (C + H^\infty)(\mathbb{T})$, then the following holds for the Toeplitz operator $T(a)  : H[X(w)] \to H[X(w)]$, 
\begin{equation}\label{utransl}
\|a\|_{L^\infty} = \|T(a)\|_m \le \|T(a)\|_{\mathrm{e}}  \le 2 \|a\|_{L^\infty} .
\end{equation}
\end{theorem}
\begin{proof}
Since $M^*(H[X(w)]) \le 2$ (see Lemma \ref{transll}),
the same argument as in the proof of Theorem \ref{upperest} shows that
$$
\|T(a)\|_{\mathrm{e}}  \le 2 \|T(a)\|_m \le 2 \|a\|_{L^\infty} .
$$
According to Lemma \ref{HWSc},  $\|a\|_{L^\infty}$ is less than or equal to
the essential spectral radius of $T(a)$, while the latter is less than or equal to $\|T(a)\|_m$ (see \cite[$\S$6]{LS}).
Hence $\|a\|_{L^\infty} \le \|T(a)\|_m$.
\end{proof}

\begin{theorem}\label{lowerestX} 
Let $X$ be a separable rearrangement-invariant Banach function space (see \cite[Ch. 2]{BS88}). 
If the Riesz projection is bounded on $X$(see \cite[Ch. 2, Lemmas 6.6, 6.7, and Corollary 6.11]{BS88}), then
the following equalities hold for the Toeplitz operator $T(\mathbf{e}_{-1}) : H[X] \to H[X]$,
\begin{equation}\label{lowerX}
\|T(\mathbf{e}_{-1})\|_\chi =  \|T(\mathbf{e}_{-1})\|_{\mathrm{e}} = \|T(\mathbf{e}_{-1})\|_{H[X]  \to H[X]}\, .
\end{equation}
\end{theorem}
\begin{proof}
The proof is only a slight modification of that of Theorem \ref{lowerest}.
It is sufficient to prove that $\|T(\mathbf{e}_{-1})\|_\chi \ge \|T(\mathbf{e}_{-1})\|_{H[X]  \to H[X]} =: C_X$.
For any $\varepsilon > 0$, there exists $q \in H[X]$,
$$
q\left(e^{i\theta}\right) = \sum_{k = 0}^\infty c_k e^{i k\theta} = c_0 + \sum_{k = 1}^\infty c_k e^{i k\theta} =: c_0 + q_0\left(e^{i\theta}\right) , 
\ \ \ \theta \in [-\pi, \pi] , 
$$
such that $\|q\|_{H[X]} = 1$ and  
$$
\|q_0\|_{H[X]} = \|\mathbf{e}_{-1}q_0\|_{H[X]} = \|T(\mathbf{e}_{-1})q\|_{H[X]} \ge C_X - \varepsilon .
$$
Since $\mathbf{e}_N$, $N \in \mathbb{N}$ is an inner function and $\mathbf{e}_N(0) = 0$, $\mathbf{e}_N$  is a 
measure-preserving transformation from $\mathbb{T}$ onto itself
 (\cite[Lemma 1]{N68}, see also \cite[Remark 9.4.6]{CMR06} and \cite[Lemma 2.5]{KS19}).
 Since $X$ is rearrangement-invariant, $\|f\circ\mathbf{e}_N\|_{X} = \|f\|_{X}$ for every $f \in H[X] \subset H^1(\mathbb{T})$. 
 For any such $f$, one has $f\circ\mathbf{e}_N \in H^1(\mathbb{T})$ as in the proof of Theorem \ref{lowerest}. Hence
 $f\circ\mathbf{e}_N \in X \cap H^1(\mathbb{T}) = H[X]$. So, $\|q\circ\mathbf{e}_N\|_{H[X]} = 1$ and 
$\|q_0\circ\mathbf{e}_N\|_{H[X]} = \|q_0\|_{H[X]}$.

Take any finite set $\{\varphi_1, \dots, \varphi_m\} \subset H[X]$. Since $X$ is separable, there exist polynomials
$$
\psi_j(z) := \sum_{k = 0}^{n_j} \psi_{j, k} z^k , \ \ \ z \in \mathbb{C} , \ \ j = 1, \dots, m
$$
such that $\|\varphi_j - \psi_j\|_{H[X]} \le \varepsilon$ (see \cite[Theorem 1.1]{K17}, \cite[Lemma 3.4(c)]{L19}, or \cite[Theorem 1.5]{KS18}). 
Then choose $N \in \mathbb{N}$ such that
\begin{equation}\label{NX}
N > \max\{n_1, \dots, n_m\} + 1 .
\end{equation}
There exists $h_0 \in X'$ such that $\|h_0\|_{X'} = 1$ and
$$
\int_{-\pi}^\pi q_0\left(e^{i\theta}\right) h_0\left(e^{i\theta}\right)\, d\theta \ge \|q_0\|_{X} - \varepsilon
$$
(see \cite[Ch. 1, Theorem 2.9]{BS88}).
Set $h := h_0\circ\mathbf{e}_N$. Since $X'$ is rearrangement-invariant (see \cite[Ch. 2, Proposition 4.2]{BS88}), one has, as above,
$\|h\|_{X'} = \|h_0\|_{X'} =1$ and 
$$
\int_{-\pi}^\pi (q_0\circ\mathbf{e}_N)\left(e^{i\theta}\right) h\left(e^{i\theta}\right)\, d\theta 
= \int_{-\pi}^\pi q_0\left(e^{i\theta}\right) h_0\left(e^{i\theta}\right)\, d\theta \ge \|q_0\|_{X} - \varepsilon .
$$
The Fourier series of $h$ has the form
$$
\sum_{k \in \mathbb{Z}} h_k e^{ik N\theta} , \ \ \ h_k \in \mathbb{C} .
$$
It follows from \eqref{NX} that
$$
\{k N | \ k \in \mathbb{Z}\} \cap \{1, \dots , n_j + 1\} = \emptyset , \ \ \ j = 1, \dots, m .
$$
Hence
$$
\int_{-\pi}^\pi (\mathbf{e}_1 \psi_j)\left(e^{i\theta}\right) h\left(e^{i\theta}\right)\, d\theta = 0 .
$$
So,
$$
\int_{-\pi}^\pi (q_0\circ\mathbf{e}_N - \mathbf{e}_1 \psi_j)\left(e^{i\theta}\right) h\left(e^{i\theta}\right)\, d\theta \ge \|q_0\|_{X} - \varepsilon ,
\quad j = 1, \dots, m .
$$
On the other hand, H\"older's inequality (see \cite[Ch. 1, Theorem 2.4]{BS88}) implies 
$$
\left|\int_{-\pi}^\pi (q_0\circ\mathbf{e}_N - \mathbf{e}_1 \psi_j)\left(e^{i\theta}\right) h\left(e^{i\theta}\right)\, d\theta\right| \le 
\|q_0\circ\mathbf{e}_N - \mathbf{e}_1 \psi_j\|_{X}, 
$$
since $\|h\|_{X'} = 1$. Hence
$$
\|q_0\circ\mathbf{e}_N - \mathbf{e}_1 \psi_j\|_{X} \ge \|q_0\|_{X} - \varepsilon = \|q_0\|_{H[X]} - \varepsilon
$$
and
\begin{align*}
\|T(\mathbf{e}_{-1})(q\circ\mathbf{e}_N) - \varphi_j\|_{X} = \|\mathbf{e}_{-1}(q_0\circ\mathbf{e}_N) - \varphi_j\|_{X}
= \|q_0\circ\mathbf{e}_N - \mathbf{e}_1\varphi_j\|_{X} \\
\ge \|q_0\circ\mathbf{e}_N - \mathbf{e}_1 \psi_j\|_{X} - \varepsilon \ge \|q_0\|_{H[X]} - 2\varepsilon \ge C_X - 3\varepsilon ,
 \ \ j = 1, \dots, m .
\end{align*}
So, for every finite set $\{\varphi_1, \dots, \varphi_m\} \subset H[X]$, there exist an element of the image of the unit ball
$T(\mathbf{e}_{-1})\left(B_{H[X]}\right)$ that lies at a distance at least $C_X - 3\varepsilon$ from every element of $\{\varphi_1, \dots, \varphi_m\}$.
This means that $T(\mathbf{e}_{-1})\left(B_{H[X]}\right)$ cannot 
be covered by a finite family of open balls of radius $C_X - 3\varepsilon$. Hence
$$
\|T(\mathbf{e}_{-1})\|_\chi \ge C_X - 3\varepsilon = \|T(\mathbf{e}_{-1})\|_{H[X] \to H[X]} - 3 \varepsilon , \quad \forall \varepsilon > 0 ,
$$
i.e. $\|T(\mathbf{e}_{-1})\|_\chi \ge \|T(\mathbf{e}_{-1})\|_{H[X] \to H[X]}$.
\end{proof}

It would be interesting to know for which (rearrangement-invariant) Banach function spaces the inequality $\|T(\mathbf{e}_{-1})\|_{H[X] \to H[X]} > 1$ hods.
One has
\begin{align*}
\|T(\mathbf{e}_{-1})\|_{H[X] \to H[X]} & = \|\mathbf{e}_{-1} (I - \mathbf{ K}_0)\|_{H[X] \to H[X]} = \|I - \mathbf{ K}_0\|_{H[X] \to H[X]} \\
& \le \|I - \mathbf{ K}_0\|_{X \to X} ,
\end{align*}
and it is known that $\|I - \mathbf{ K}_0\|_{X \to X} > 1$ for every separable rearrangement-invariant Banach function space $X$
not isometric to $L^2(\mathbb{T})$ (see \cite[Theorem 4]{R95}).  However, the latter does not immediately imply that 
$\|T(\mathbf{e}_{-1})\|_{H[X] \to H[X]} > 1$, since the inequality $\|I - \mathbf{ K}_0\|_{H[X] \to H[X]} \le \|I - \mathbf{ K}_0\|_{X \to X}$ is, in general, strict. 
Indeed, $\|I - \mathbf{ K}_0\|_{L^p \to L^p} \to 2$ as $p \to 1$ (see \cite{F90} and \cite{M09}), while $\|I - \mathbf{ K}_0\|_{H^p \to H^p} < 1.71$
for all $p$ sufficiently close to $1$ (see \cite{F17}).

\section{Appendix:  Adjoints to restrictions of operators}

Here we present some well known results for which we could not find a convenient reference. We have used \eqref{adj} in the proof of Lemma \ref{transll}.

Let $X$ and $Y$ be Banach spaces, $X_0 \subseteq X$ and $Y_0 \subseteq Y$ be closed linear subspaces, and
let $A \in \mathcal{B}(X, Y)$ be such that $A(X_0) \subseteq Y_0$. Let $A_0 \in \mathcal{B}(X_0, Y_0)$ be the restriction of $A$ to $X_0$:
$$
A_0 x_0 := Ax_0 \in Y_0 \ \mbox{ for all } \ x_0 \in X_0 .
$$
Let
$$
X_0^\perp := \{x^* \in X^* : \ x^*(x_0) = 0  \mbox{ for all } \ x_0 \in X_0\}
$$
and let $Y_0^\perp$ be defined similarly. Then $X_0^*$ and $Y_0^*$ are isometrically isomorphic to the quotient spaces
$X^*/X_0^\perp$ and $Y^*/Y_0^\perp$, respectively (see, e.g., \cite[Theorem 7.1]{D70}). We will identify these spaces and will denote by $[x^*]$ the element of
$X^*/X_0^\perp$ corresponding to $x^* \in X^*$, and similarly for $y^* \in Y^*$.
 
It is easy to see that $A^*(Y_0^\perp) \subseteq X_0^\perp$. Indeed, take any $y^*_0 \in Y_0^\perp$ and $x_0 \in X_0$.
Since $Ax_0 \in Y_0$, one has
$$
(A^*y^*_0)(x_0) = y^*_0(Ax_0) = 0 .
$$
So, $A^*y^*_0 \in X_0^\perp$. Hence the operator $[A^*]$,
$$
[A^*] [y^*] := [A^* y^*] \in X^*/X_0^\perp , \quad [y^*] \in Y^*/Y_0^\perp
$$
is a well defined element of $\mathcal{B}(Y^*/Y_0^\perp, X^*/X_0^\perp) = \mathcal{B}(Y^*_0, X^*_0)$, and it is easy to see that $A_0^* = [A^*]$.
Indeed, one has for every $[y^*] \in Y^*/Y_0^\perp$ and $x_0 \in X_0$
\begin{align*}
(A_0^*[y^*])(x_0) & = [y^*](A_0x_0) = [y^*](Ax_0) = y^*(Ax_0) = (A^*y^*)(x_0) \\
& = [A^*y^*](x_0) = ([A^*] [y^*])(x_0).
\end{align*}
Finally,
\begin{align}\label{adj}
\|A_0^*[y^*]\|_{X_0^*} & = \|A_0^*[y^*]\|_{X^*/X_0^\perp} =  \|[A^*][y^*]\|_{X^*/X_0^\perp} =  \|[A^*y^*]\|_{X^*/X_0^\perp} \nonumber \\
& = \inf_{x_0 \in X_0} \|A^*y^* + x_0\|_{X^*} \le \|A^*y^*\|_{X^*}\, .
\end{align}

\textit{Address:} \\
Department of Mathematics,
King's College London\\
Strand,
London WC2R 2LS,
United Kingdom \\
and \\
Technische Universit\"at Dresden,
Fakult\"at Mathematik\\
01062 Dresden,
Germany

\textit{E-mail:} eugene.shargorodsky@kcl.ac.uk

\end{document}